\documentclass[reqno,12pt]{amsart}
\usepackage{amsfonts}

\usepackage{amsmath}
\usepackage{graphicx}
\usepackage{amscd}

\newtheorem{theorem}{Theorem}
\theoremstyle{plain}
\newtheorem{acknowledgement}{Acknowledgement}

\newtheorem{corollary}{Corollary}

\newtheorem{remark}{Remark}

\numberwithin{equation}{section}

\input{tcilatex}

\begin{document}
\author{}
\title{}
\maketitle

\begin{center}
\thispagestyle{empty}\textbf{{\Large {COMPLETE SUM OF PRODUCTS OF }}}$%
\mathbf{{\Large {(h,q)}}}$\textbf{{\Large {-EXTENSION OF EULER \textbf{%
POLYNOMIALS AND }NUMBERS}}}

\bigskip

\textbf{{\Large {Yilmaz\ SIMSEK}}}

\bigskip

\textit{{\Large {University of Akdeniz, Faculty of Arts and Science,
Department of Mathematics,}}}{\Large {\textit{\ 07058 Antalya}, Turkey}}

{\Large %\textit{\Large{Canada}}
}

{\Large \textbf{E-Mail: simsek@akdeniz.edu.tr}}

{\Large \ }

\textbf{{\Large {Abstract}}}
\end{center}

By using the fermionic $p$-adic $q$-Volkenborn integral, we construct
generating functions of higher-order $\left( h,q\right) $-extension of Euler
polynomials and numbers. By using these numbers and polynomials, we give new
approach to the complete sums of products of $(h,q)$-extension of Euler
polynomials and numbers one of which is given by the following form:%
\begin{equation*}
E_{n,q}^{(h,v)}(y_{1}+y_{2}+...+y_{v})=\sum_{%
\begin{array}{c}
l_{1},l_{2},...,l_{v}\geq 0 \\ 
l_{1}+l_{2}+...+l_{v}=n%
\end{array}%
}\left( 
\begin{array}{c}
n \\ 
l_{1},l_{2},...,l_{v}%
\end{array}%
\right) \prod_{j=1}^{v}E_{l_{j},q}^{(h)}(y_{j}),
\end{equation*}%
where $\left( 
\begin{array}{c}
n \\ 
l_{1},l_{2},...,l_{v}%
\end{array}%
\right) $ are the multinomial coefficients and $E_{m,q}^{(h)}(y)$ is the $%
(h,q)$-extension of Euler polynomials. Furhermore, we define some identities
involving $(h,q)$-extension of Euler polnomials and numbers.



\section{Introduction, Definitions and Notations}

The main aim of this paper is to study higher-order $(h,q)$-extension of
Euler numbers and polynomials. Bernoulli and Euler numbers and polynomials
were studied by many authors (see for detail \cite{hacerSimsek}, \cite%
{hacersimsekcangul}, \cite{Kim16}, \cite{Kim2002}, \cite{Kim2006}, \cite%
{Kim2007b}, \cite{kimJCAAA2007}, \cite{kimjnmp2007}, \cite{Kim-Rim2007}, %
\cite{kimRJMPEuler}, \cite{simBKMS}, \cite{simJNT2005}, \cite{Simsek2006a}, %
\cite{Srivastava}, \cite{TKIMmodf}) . We introduce some of them here. In %
\cite{Kim2002}, \cite{kimRJMPEuler}, \cite{Kim2006}, Kim construced $p$-adic 
$q$-Volkenborn integral identities. By using these identities, he proved $p$%
-adic $q$-integral representation of $q$-Euler and Bernoulli numbers\ and
polynomials. In \cite{simBKMS}, \cite{simJNT2005}, we constructed generating
functions of $q$-generalized Euler numbers and polynomials and twisted\ $q$%
-generalized Euler numbers and polynomials. We also constructed a complex
analytic twisted $l$-series which is interpolated twisted $q$-Euler numbers
at non-positive integers. In \cite{Kim16}, by using $q$-Volkenborn
integration, Kim constructed the new $(h,q)$-extension of the Bernoulli
numbers and polynomials. He defined $(h,q)$-extension of the zeta functions
which are interpolated new $(h,q)$-extension of the Bernoulli numbers and
polynomials. In \cite{Simsek2006a}, we defined twisted $(h,q)$-Bernoulli
numbers, zeta functions and $L$-function. He also gave relations between
these functions and numbers.

By the same motivation of the above studies, in this paper, we construct new
approach to the complete sums of products of $(h,q)$-extension of Euler
polynomials and numbers. By the Multinomial Theorem and multinomial
relations, we find new identities related to these polynomials and numbers.

Throughout this paper $\mathbb{Z},$ $\mathbb{Z}_{p},$ $\mathbb{Q}_{p}$ and $%
\mathbb{C}_{p}$ will be denoted by the ring of rational integers, the ring
of $p$-adic integers, the field of $p$-adic rational numbers and the
completion of the algebraic closure of $\mathbb{Q}_{p},$ respectively. Let $%
v_{p}$ be the normalized exponential valuation of $\mathbb{C}_{p}$ with $%
\mid p\mid _{p}=p^{-v_{p}(p)}=p^{-1}.$ When one talks of $q$-extension, $q$
is variously considered as an indeterminate, a complex number $q\in \mathbb{C%
}$, or $p$-adic number $q\in \mathbb{C}_{p}.$\ If $q\in \mathbb{C}_{p},$
then we normally assume 
\begin{equation*}
\mid q-1\mid _{p}<p^{-\frac{1}{p-1}},
\end{equation*}%
so that%
\begin{equation*}
q^{x}=\exp (x\log q)\text{ for}\mid x\mid _{p}\leq 1.
\end{equation*}%
If $q\in \mathbb{C}$, then we normally assume$\mid q\mid <1$. cf. (\cite%
{Kim2002}, \cite{jang}, \cite{Kim2006}, \cite{Kim2007b}, \cite{kimJCAAA2007}%
, \cite{kimRJMPEuler}, \cite{Simsek2006a}).We use the notations as%
\begin{equation*}
\left[ x\right] _{q}=\frac{1-q^{x}}{1-q},\text{ \ }\left[ x\right] _{-q}=%
\frac{1-\left( -q\right) ^{x}}{1+q}.
\end{equation*}%
Let $UD\left( \mathbb{Z}_{p}\right) $ be the set of uniformly differentiable
function on $\mathbb{Z}_{p}$. For $f\in UD\left( \mathbb{Z}_{p}\right) $, Kim%
\cite{Kim2002} defined the $p$-adic invariant $q$-integral on $\mathbb{Z}%
_{p} $ as follows:%
\begin{equation*}
I_{q}\left( f\right) =\int\limits_{\mathbb{Z}_{p}}f\left( x\right) d\mu
_{q}\left( x\right) =\underset{N\rightarrow \infty }{\lim }\frac{1}{\left[
p^{N}\right] _{q}}\sum_{x=0}^{p^{N}-1}f\left( x\right) q^{x},
\end{equation*}%
where $N$ is a fixed natural number. The bosonic integral was considered
from a physical point of view to the bosonic limit $q\rightarrow 1$, as
follows:%
\begin{equation*}
I_{1}\left( f\right) =\underset{q\rightarrow 1}{\lim }I_{q}\left( f\right) 
\text{, cf.\ (\cite{Kim2002}, \cite{Kim2006}, \cite{Kim2007b}, , \cite%
{kimRimJMAA2007}).}
\end{equation*}%
We consider the fermionic integral in contrast to the conventional bosonic,
which is called the $q$-deformed fermionic integral on $\mathbb{Z}_{p}.$
That is%
\begin{equation*}
I_{-q}\left( f\right) =\underset{q\rightarrow -q}{\lim }I_{q}\left( f\right)
=\int_{\mathbb{Z}_{p}}f\left( x\right) d\mu _{-q}\left( x\right) \text{,
cf.\ (\cite{Kim2006}, \cite{Kim2007b}, \cite{kimRJMPEuler}, \cite%
{kimRimJMAA2007}).}
\end{equation*}%
Recently, twisted $\left( h,q\right) $-Bernoulli and Euler numbers and
polynomials were studied by several authors cf.(see \cite%
{Kim-Jang-Rim-Pak2003}, \cite{jang}, \cite{simBKMS}, \cite{simJNT2005}, \cite%
{Simsek2006a}, \cite{TKIMmodf}, \cite{kimjnmp2007}, \cite{Kim-Rim2007}).

By definition of $\mu _{-q}(x)$, we see that 
\begin{equation}
I_{-1}(f_{1})+I_{-1}(f)=2f(0),\text{ cf. \cite{Kim2006},}  \label{5}
\end{equation}%
where $f_{1}(x)=f(x+1)$.

In \cite{hacerSimsek}, Ozden and Simsek defined new $(h,q)$-extension of
Euler numbers and polynomials. By using derivative operator to these
functions, we derive $\left( h,q\right) $-extension of zeta functions and $l$%
-functions, which interpolate $(h,q)$-extension of Euler numbers at negative
integers.

$(h,q)$-extension of Euler polynomials, $E_{n,q}^{(h)}(x)$ are defined by%
\begin{equation}
F_{q}^{h}(t,x)=F_{q}^{h}(t)e^{tx}=\frac{2e^{tx}}{q^{h}e^{t}+1}%
=\sum_{n=0}^{\infty }E_{n,q}^{(h)}(x)\frac{t^{n}}{n!}\text{ cf. \cite%
{hacerSimsek}.}  \label{aa10}
\end{equation}%
For $x=0$, we have%
\begin{equation}
F_{q}^{h}(t)=\frac{2e^{tx}}{q^{h}e^{t}+1}=\sum_{n=0}^{\infty }E_{n,q}^{(h)}%
\frac{t^{n}}{n!}\text{ cf. \cite{hacerSimsek}.}  \label{11}
\end{equation}

\begin{theorem}
(\cite{hacerSimsek})\label{theo01a}(Witt formula)For $h\in 
%TCIMACRO{\U{2124} }%
%BeginExpansion
\mathbb{Z}
%EndExpansion
$, $q\in \mathbb{C}_{p}$ with $\left| 1-q\right| _{p}<p^{-\frac{1}{p-1}}$.%
\begin{equation}
\int_{\mathbb{Z}_{p}}q^{hx}x^{n}d\mu _{-1}(x)=E_{n,q}^{(h)},  \label{10}
\end{equation}%
\begin{equation}
\int_{\mathbb{Z}_{p}}q^{hy}(x+y)^{n}d\mu _{-1}(y)=E_{n,q}^{(h)}(x).
\label{10-1}
\end{equation}
\end{theorem}

\begin{theorem}
(\cite{hacerSimsek})\label{teo2}(Distribution Relation) For $d$ is an odd
positive integer, $k\in \mathbb{N}$, we have 
\begin{equation*}
E_{k,q}^{(h)}(x,q)=d^{k}\sum_{a=0}^{d-1}(-1)^{a}q^{ha}E_{k,q^{d}}^{(h)}%
\left( \frac{x+a}{d}\right) .
\end{equation*}
\end{theorem}

\section{Higher-order $(h,q)$-Euler polynomials and numbers}

Our main purpose in this section is to give complete sums of products of $%
(h,q)$-Euler polynomials and numbers. We define $(h,q)$-Euler polynomials
and numbers of higher-order, $E_{n,q}^{(h,k)}(z)$ and $E_{n,q}^{(h,v)}\in 
\mathbb{C}_{p}$, respectively, by making use multiple of\ $p$-adic $q$%
-integral on $\mathbb{Z}_{p}$ in the fermionic sense:%
\begin{eqnarray}
&&\underset{v-times}{\underbrace{\int_{\mathbb{Z}_{p}}...\int_{\mathbb{Z}%
_{p}}}}q^{\sum_{j=1}^{v}hx_{j}}\exp \left( t\sum_{j=1}^{v}x_{j}\right)
\prod_{j=1}^{v}d\mu _{-1}(x_{j})  \notag \\
&=&\sum_{n=0}^{\infty }E_{n,q}^{(h,v)}\frac{t^{n}}{n!},  \label{askk-1}
\end{eqnarray}%
where 
\begin{equation*}
\prod_{j=1}^{v}d\mu _{-1}(x_{j})=d\mu _{-1}(x_{1})d\mu _{-1}(x_{2})...d\mu
_{-1}(x_{v}).
\end{equation*}%
By using Taylor series of $\exp (tx)$ in the above equation, we have 
\begin{eqnarray*}
&&\sum_{n=0}^{\infty }\left( \int_{\mathbb{Z}_{p}}...\int_{\mathbb{Z}%
_{p}}q^{\sum_{j=1}^{v}hx_{j}}\left( \sum_{j=1}^{v}x_{j}\right)
^{n}\prod_{j=1}^{v}d\mu _{-1}(x_{j})\right) \frac{t^{n}}{n!} \\
&=&\sum_{n=0}^{\infty }E_{n,q}^{(h,k)}\frac{t^{n}}{n!}\text{.}
\end{eqnarray*}%
By comparing coefficients $\frac{t^{n}}{n!}$ in the above equation, we
arrive at the following theorem.

\begin{theorem}
\label{new-theorem1}For positive integers $n$, $v$, and $h\in \mathbb{Z}$,we
have%
\begin{equation}
E_{n,q}^{(h,v)}=\int_{\mathbb{Z}_{p}}...\int_{\mathbb{Z}_{p}}q^{%
\sum_{j=1}^{v}hx_{j}}\left( \sum_{j=1}^{v}x_{j}\right)
^{n}\prod_{j=1}^{v}d\mu _{-1}(x_{j}).  \label{askk-1a}
\end{equation}
\end{theorem}

By (\ref{askk-1}), $(h,q)$-Euler numbers of higher-order, $E_{n,q}^{(h,v)}$
are defined by means of the following generating function 
\begin{equation*}
\left( \frac{2}{q^{h}e^{t}+1}\right) ^{v}=\sum_{n=0}^{\infty }E_{n,q}^{(h,v)}%
\frac{t^{n}}{n!}\text{.}
\end{equation*}

Observe that for $v=1$, the above equation reduces to (\ref{11}).

\begin{eqnarray*}
&&\underset{v-times}{\underbrace{\int_{\mathbb{Z}_{p}}...\int_{\mathbb{Z}%
_{p}}}}q^{\sum_{j=1}^{v}hx_{j}}\exp \left( tz+\sum_{j=1}^{v}tx_{j}\right)
\prod_{j=1}^{v}d\mu _{-1}(x_{j}) \\
&=&\sum_{n=0}^{\infty }E_{n,q}^{(h,v)}(z)\frac{t^{n}}{n!}\text{.}
\end{eqnarray*}%
By using Taylor series of $\exp (tx)$ in the above equation, we have 
\begin{eqnarray*}
&&\sum_{n=0}^{\infty }\left( \int_{\mathbb{Z}_{p}}...\int_{\mathbb{Z}%
_{p}}q^{\sum_{j=1}^{v}hx_{j}}\left( z+\sum_{j=1}^{v}x_{j}\right)
^{n}\prod_{j=1}^{v}d\mu _{-1}(x_{j})\right) \frac{t^{n}}{n!} \\
&=&\sum_{n=0}^{\infty }E_{n,q}^{(h,v)}(z)\frac{t^{n}}{n!}\text{.}
\end{eqnarray*}%
By comparing coefficients $\frac{t^{n}}{n!}$ in the above equation, we
arrive at the following theorem.

\begin{theorem}
\label{new-theorem2}For $z\in \mathbb{C}_{p}$ and positive integers $n$, $v$%
, and $h\in \mathbb{Z}$, we have%
\begin{equation}
E_{n,q}^{(h,v)}(z)=\int_{\mathbb{Z}_{p}}...\int_{\mathbb{Z}%
_{p}}q^{\sum_{j=1}^{v}hx_{j}}\left( z+\sum_{j=1}^{v}x_{j}\right)
^{n}\prod_{j=1}^{v}d\mu _{-1}(x_{j})  \label{askk-1b}
\end{equation}
\end{theorem}

By (\ref{askk-1}), $(h,q)$-Euler polynomials of higher-order, $%
E_{n,q}^{(h,v)}(z)$ are defined by means of the following generating
function 
\begin{equation*}
F_{q,w}^{h,v}(z,t)=e^{tz}\left( \frac{2}{q^{h}e^{t}+1}\right)
^{v}=\sum_{n=0}^{\infty }E_{n,q}^{(h,v)}(z)\frac{t^{n}}{n!}\text{.}
\end{equation*}%
Note that when $v=1$, then we have (\ref{aa10}), when $h=1$, $q\rightarrow 1$%
, then we have%
\begin{equation*}
F^{1,v}(z,t)=e^{tz}\left( \frac{2}{e^{t}+1}\right) ^{v}=\sum_{n=0}^{\infty
}E_{n}^{(v)}(z)\frac{t^{n}}{n!}\text{,}
\end{equation*}%
where $E_{n}^{(v)}(z)$ are denoted classical higher-order Euler polynomials
cf. \cite{Srivastava}.

\begin{theorem}
\label{new-theorem3}For $z\in \mathbb{C}_{p}$ and positive integers $n$, $v$%
, and $h\in \mathbb{Z}$, we have%
\begin{equation}
E_{n,q}^{(h,v)}(z)=\sum_{l=0}^{n}\left( 
\begin{array}{c}
n \\ 
l%
\end{array}%
\right) z^{n-l}E_{l,q}^{(h,v)}.  \label{askk-1a0}
\end{equation}
\end{theorem}

\begin{proof}
By using binomial expansion in (\ref{askk-1b}), we have%
\begin{equation*}
E_{n,q}^{(h,v)}(z)=\sum_{l=0}^{n}\left( 
\begin{array}{c}
n \\ 
l%
\end{array}%
\right) z^{n-l}\int_{\mathbb{Z}_{p}}...\int_{\mathbb{Z}_{p}}q^{%
\sum_{j=1}^{v}hx_{j}}\left( \sum_{j=1}^{v}x_{j}\right)
^{l}\prod_{j=1}^{v}d\mu _{-1}(x_{j}).
\end{equation*}%
By (\ref{askk-1a}) in the above, we arrive at the desired result.
\end{proof}

\section{The complete sums of products of $(h,q)$-extension of Euler
polynomials and numbers}

In this section, we prove main theorems on the complete sums of products of $%
(h,q)$-extension of Euler polynomials and numbers. Firstly, we need the
Multinomial Theorem, which is given as follows cf. (\cite{comtet}, \cite%
{knuth}):

\begin{theorem}
\label{theoremMultiNomial}(Multinomial Theorem)%
\begin{equation*}
\left( \sum_{j=1}^{v}x_{j}\right) ^{n}=\sum_{%
\begin{array}{c}
l_{1},l_{2},...,l_{v}\geq 0 \\ 
l_{1}+l_{2}+...+l_{v}=n%
\end{array}%
}\left( 
\begin{array}{c}
n \\ 
l_{1},l_{2},...,l_{v}%
\end{array}%
\right) \prod_{a=1}^{v}x_{a}^{l_{a}},
\end{equation*}%
where $\left( 
\begin{array}{c}
n \\ 
l_{1},l_{2},...,l_{v}%
\end{array}%
\right) $ are the multinomial coefficients, which are defined by%
\begin{equation*}
\left( 
\begin{array}{c}
n \\ 
l_{1},l_{2},...,l_{v}%
\end{array}%
\right) =\frac{n!}{l_{1}!l_{2}!...l_{v}!}.
\end{equation*}
\end{theorem}

\begin{theorem}
\label{new-theorem4}For positive integers $n$, $v$, we have%
\begin{equation}
E_{n,q}^{(h,v)}=\sum_{%
\begin{array}{c}
l_{1},l_{2},...,l_{v}\geq 0 \\ 
l_{1}+l_{2}+...+l_{v}=n%
\end{array}%
}\left( 
\begin{array}{c}
n \\ 
l_{1},l_{2},...,l_{v}%
\end{array}%
\right) \prod_{j=1}^{v}E_{l_{j},q}^{(h)},  \label{askk-2a}
\end{equation}%
where $\left( 
\begin{array}{c}
n \\ 
l_{1},l_{2},...,l_{v}%
\end{array}%
\right) $ is the multinomial coefficient.
\end{theorem}

\begin{proof}
By using Theorem \ref{theoremMultiNomial} in (\ref{askk-1a}), we have%
\begin{eqnarray*}
&&E_{n,q}^{(h,v)} \\
&=&\sum_{%
\begin{array}{c}
l_{1},l_{2},...,l_{v}\geq 0 \\ 
l_{1}+l_{2}+...+l_{v}=n%
\end{array}%
}\left( 
\begin{array}{c}
n \\ 
l_{1},l_{2},...,l_{v}%
\end{array}%
\right) \prod_{j=1}^{v}\int_{\mathbb{Z}_{p}}q^{hx_{j}}x_{j}^{l_{j}}d\mu
_{-1}(x_{j}).
\end{eqnarray*}%
By (\ref{10}) in the above, we obtain the desired result.
\end{proof}

\begin{remark}
\begin{equation*}
\lim_{%
\begin{array}{c}
q\rightarrow 1 \\ 
h=1%
\end{array}%
}E_{n,q}^{(h,v)}=E_{n}^{(v)}=\sum_{%
\begin{array}{c}
l_{1},l_{2},...,l_{v}\geq 0 \\ 
l_{1}+l_{2}+...+l_{v}=n%
\end{array}%
}\left( 
\begin{array}{c}
n \\ 
l_{1},l_{2},...,l_{v}%
\end{array}%
\right) E_{l_{1}}E_{l_{2}}...E_{l_{v}},
\end{equation*}%
where $\left( 
\begin{array}{c}
n \\ 
l_{1},l_{2},...,l_{v}%
\end{array}%
\right) $ is the multinomial coefficient, and $E_{l_{j}}$, $1\leq j\leq v$,
are denoted classical Euler numbers.
\end{remark}

By substituting (\ref{askk-2a}) into (\ref{askk-1a0}), after some elementary
calculations, we arrive at the following corollary:

\begin{corollary}
\label{corollary1}For $z\in \mathbb{C}_{p}$ and positive integers $n$, $v$,
we have%
\begin{equation*}
E_{n,q}^{(h,v)}(z)=\sum_{m=0}^{n}\sum_{%
\begin{array}{c}
l_{1},l_{2},...,l_{v}\geq 0 \\ 
l_{1}+l_{2}+...+l_{v}=m%
\end{array}%
}\left( 
\begin{array}{c}
n \\ 
m%
\end{array}%
\right) \left( 
\begin{array}{c}
m \\ 
l_{1},l_{2},...,l_{v}%
\end{array}%
\right) z^{n-m}\prod_{j=1}^{v}E_{l_{j},q}^{(h)}.
\end{equation*}%
where $\left( 
\begin{array}{c}
n \\ 
l_{1},l_{2},...,l_{v}%
\end{array}%
\right) $ and $\left( 
\begin{array}{c}
n \\ 
m%
\end{array}%
\right) $ are the multinomial coefficient and the binomial coefficient,
respectively.
\end{corollary}

One of the main theorem of this section is to give complete sum of products
of $(h,q)$-Euler polynomials.

\begin{theorem}
\label{new-theorem5}For $y_{1}$, $y_{2}$,..., $y_{v}\in \mathbb{C}_{p}$ and
positive integers $n$, $v$, we have%
\begin{equation}
E_{n,q}^{(h,v)}(y_{1}+y_{2}+...+y_{v})=\sum_{%
\begin{array}{c}
l_{1},l_{2},...,l_{v}\geq 0 \\ 
l_{1}+l_{2}+...+l_{v}=n%
\end{array}%
}\left( 
\begin{array}{c}
n \\ 
l_{1},l_{2},...,l_{v}%
\end{array}%
\right) \prod_{j=1}^{v}E_{l_{j},q}^{(h)}(y_{j}),  \label{ay-1}
\end{equation}%
where $\left( 
\begin{array}{c}
n \\ 
l_{1},l_{2},...,l_{v}%
\end{array}%
\right) $ is the multinomial coefficient.
\end{theorem}

\begin{proof}
By substituting $z=y_{1}+y_{2}+...+y_{v}$ into (\ref{askk-1b}), we have%
\begin{equation*}
E_{n,q}^{(h,v)}(y_{1}+y_{2}+...+y_{v})=\int_{\mathbb{Z}_{p}}...\int_{\mathbb{%
Z}_{p}}q^{\sum_{j=1}^{v}hx_{j}}\left( \sum_{j=1}^{v}(y_{j}+x_{j})\right)
^{n}\prod_{j=1}^{v}d\mu _{-1}(x_{j}).
\end{equation*}%
By using Theorem \ref{theoremMultiNomial} in the above, and after some
elementary calculations, we get%
\begin{eqnarray*}
&&E_{n,q}^{(h,v)}(y_{1}+y_{2}+...+y_{v}) \\
&=&\sum_{%
\begin{array}{c}
l_{1},l_{2},...,l_{v}\geq 0 \\ 
l_{1}+l_{2}+...+l_{v}=n%
\end{array}%
}\left( 
\begin{array}{c}
n \\ 
l_{1},l_{2},...,l_{v}%
\end{array}%
\right) \prod_{j=1}^{v}\int_{\mathbb{Z}_{p}}q^{hx_{j}}(y_{j}+x_{j})^{l_{j}}d%
\mu _{-1}(x_{j}).
\end{eqnarray*}%
By substituting (\ref{10-1}) into the above, we arrive at the desired result.
\end{proof}

\begin{remark}
If we take $y_{1}=y_{2}=...=y_{v}=0$ in Theorem \ref{new-theorem5}, then
Theorem \ref{new-theorem5} reduces to Theorem \ref{new-theorem4}.
Substituting $h=1$ and $q\rightarrow 1$ into (\ref{ay-1}), we obtain the
following relation:%
\begin{equation*}
E_{n}^{(v)}(y_{1}+y_{2}+...+y_{v})=\sum_{%
\begin{array}{c}
l_{1},l_{2},...,l_{v}\geq 0 \\ 
l_{1}+l_{2}+...+l_{v}=m%
\end{array}%
}\left( 
\begin{array}{c}
m \\ 
l_{1},l_{2},...,l_{v}%
\end{array}%
\right) \prod_{j=1}^{v}E_{l_{j}}(y_{j}).
\end{equation*}%
I-C. Huang and S-Y. Huang\cite{Huang} found complete sums of products of
Bernoulli polynomials. Kim\cite{KimArchivMath} defined Carlitz's $q$%
-Bernoulli number of higher order using an integral by the $q$-analogue $\mu
_{q}$ of the ordinary $p$-adic invariant measure. He gave different proof of
complete sums of products of higher order $q$-Bernoulli polynomials. In \cite%
{jangParkRo}, Jang et al gave complete sums of products of Bernoulli
polynomials and Frobenious Euler polynomials. In \cite{simsek-kurt-dkim},
Simsek et al gave complete sums of products of $(h,q)$-Bernoulli polynomials
and numbers.
\end{remark}

By applying the following multinomial relations cf. (\cite{comtet} pp. 25,
56), (\cite{knuth} pp. 168)%
\begin{eqnarray*}
(x+y+z)^{n} &=&\sum_{%
\begin{array}{c}
0\leq l_{1},l_{2},l_{3}\leq n \\ 
l_{1}+l_{2}+l_{3}=n%
\end{array}%
}\frac{(l_{1}+l_{2}+l_{3})!}{l_{1}!l_{2}!l_{3}!}x^{l_{1}}y^{l_{2}}z^{l_{3}}
\\
&=&\sum_{%
\begin{array}{c}
0\leq l_{1},l_{2},l_{3}\leq n \\ 
l_{1}+l_{2}+l_{3}=n%
\end{array}%
}\left( 
\begin{array}{c}
l_{1}+l_{2}+l_{3} \\ 
l_{2}+l_{3}%
\end{array}%
\right) \left( 
\begin{array}{c}
l_{2}+l_{3} \\ 
l_{3}%
\end{array}%
\right) x^{l_{1}}y^{l_{2}}z^{l_{3}},
\end{eqnarray*}%
and%
\begin{eqnarray*}
\left( 
\begin{array}{c}
l_{1}+l_{2}+...+l_{v} \\ 
l_{1},l_{2},...,l_{v}%
\end{array}%
\right) &=&\frac{(l_{1}+l_{2}+...+l_{v})!}{l_{1}!l_{2}!...l_{v}!} \\
&=&\left( 
\begin{array}{c}
l_{1}+l_{2}+l_{3}+...+l_{v} \\ 
l_{2}+l_{3}+...+l_{v}%
\end{array}%
\right) ...\left( 
\begin{array}{c}
l_{v-1}+l_{v} \\ 
l_{v}%
\end{array}%
\right) .
\end{eqnarray*}%
and the above method, one can also evaluate the other sums related to
Bernoulli Polynomials or Euler Polynomials (see \cite{Huang}, \cite%
{KimArchivMath}, \cite{simsek-kurt-dkim}).

\begin{theorem}
\label{new-theorem6}Let $n\in \mathbb{N}$. Then we have%
\begin{equation*}
E_{n,q}^{(h,v)}(z+y)=\sum_{l=0}^{n}\left( 
\begin{array}{c}
n \\ 
l%
\end{array}%
\right) E_{l,q}^{(h,v)}(y)z^{n-l}.
\end{equation*}
\end{theorem}

\begin{proof}
\begin{eqnarray*}
E_{n,q}^{(h,v)}(z+y) &=&\left( E_{q}^{(h,v)}+z+y\right) ^{n} \\
&=&\sum_{l=0}^{n}\left( 
\begin{array}{c}
n \\ 
l%
\end{array}%
\right) E_{l,q}^{(h,v)}(y+z)^{n-l} \\
&=&\sum_{l=0}^{n}\left( 
\begin{array}{c}
n \\ 
l%
\end{array}%
\right) E_{l,q}^{(h,v)}\sum_{m=0}^{n-l}\left( 
\begin{array}{c}
n-l \\ 
m%
\end{array}%
\right) y^{m}z^{n-l-m} \\
&=&\sum_{l=0}^{n}\left( 
\begin{array}{c}
n \\ 
l%
\end{array}%
\right) E_{l,q}^{(h,v)}\sum_{m=l}^{n-l}\left( 
\begin{array}{c}
n-l \\ 
m-l%
\end{array}%
\right) y^{m-l}z^{n-m} \\
&=&\sum_{0\leq l\leq m\leq n}\left( 
\begin{array}{c}
n \\ 
l%
\end{array}%
\right) \left( 
\begin{array}{c}
n-l \\ 
m-l%
\end{array}%
\right) E_{l,q}^{(h,v)}y^{m-l}z^{n-m} \\
&=&\sum_{0\leq l\leq m\leq n}\left( 
\begin{array}{c}
n \\ 
m%
\end{array}%
\right) \left( 
\begin{array}{c}
m \\ 
l%
\end{array}%
\right) E_{l,q}^{(h,v)}y^{m-l}z^{n-m} \\
&=&\sum_{m=0}^{n}\left( 
\begin{array}{c}
n \\ 
m%
\end{array}%
\right) \left( \sum_{l=0}^{m}\left( 
\begin{array}{c}
m \\ 
l%
\end{array}%
\right) E_{l,q}^{(h,v)}y^{m-l}\right) z^{n-m},
\end{eqnarray*}%
with usual convention of symbolically replacing $E_{q}^{l(h,v)}$ by $%
E_{l,q}^{(h,v)}$. By using (\ref{askk-1a0}) in the above, we have%
\begin{equation*}
E_{n,q}^{(h,v)}(z+y)=\sum_{m=0}^{n}\left( 
\begin{array}{c}
n \\ 
m%
\end{array}%
\right) E_{m,q}^{(h,v)}(y)z^{n-m}.
\end{equation*}%
Thus the proof is completed.
\end{proof}

From Theorem\ref{new-theorem5} and Theorem\ref{new-theorem6}, after some
elementary calculations, we arrive at the following interesting result:

\begin{corollary}
Let $n\in \mathbb{N}$. Then we have%
\begin{equation*}
\sum_{m=0}^{n}\left( 
\begin{array}{c}
n \\ 
m%
\end{array}%
\right) E_{m,q}^{(h,v)}(y_{1})y_{2}^{n-m}=\sum_{%
\begin{array}{c}
l_{1},l_{2}\geq 0 \\ 
l_{1}+l_{2}=n%
\end{array}%
}\left( 
\begin{array}{c}
n \\ 
l_{1},l_{2}%
\end{array}%
\right) E_{l_{1},q}^{(h)}(y_{1})B_{l_{2},q}^{(h)}(y_{2}).
\end{equation*}
\end{corollary}

\begin{acknowledgement}
The author is supported by the research fund of Akdeniz University.
\end{acknowledgement}

\end{document}